\documentclass[11pt]{article}
\usepackage{amsfonts}
\usepackage{amsthm}
\usepackage{amssymb}
\usepackage{graphicx, enumerate}
\usepackage{amsmath}
\usepackage{latexsym}
\usepackage{longtable}
\usepackage{tabularx}
\usepackage{amsmath}
\usepackage{amsfonts}
\usepackage{amsthm}
\usepackage{setspace}
\usepackage{graphicx}
\usepackage{thm-restate}
\usepackage{float}
\usepackage{rotating}
\usepackage{tikz}
\usepackage{verbatim}

\usepackage{caption}
\usepackage{subcaption}

\usepackage{soul}

\usepackage{mathtools}
\usepackage{amsmath,amssymb,amsbsy}
\usepackage{graphicx}
\usepackage{multirow}
\usepackage{amsfonts}
\usepackage{amsthm}
\usepackage{amssymb}
\usepackage{graphicx, enumerate}
\usepackage{amsmath}
\usepackage{latexsym}
\usepackage{longtable}
\usepackage{tabularx}
\usepackage{amsmath}
\usepackage{amsfonts}
\usepackage{amsthm}
\usepackage{setspace}
\usepackage{graphicx}
\usepackage{float}
\usepackage{rotating}
\usepackage{tikz}
\usepackage{verbatim}
\usepackage{soul}
\usepackage{xcolor}
\usepackage{cite}
\usepackage{changepage}

\usepackage[normalem]{ulem}

\usepackage[main=english,czech,slovak]{babel}

\definecolor{darkpastelgreen}{rgb}{0.01, 0.75, 0.24}
\definecolor{blue-violet}{rgb}{0.54, 0.17, 0.89}
\definecolor{cadmiumorange}{rgb}{0.93, 0.53, 0.18}
\definecolor{darkcoral}{rgb}{0.8, 0.36, 0.27}
\definecolor{darkorange}{rgb}{1.0, 0.55, 0.0}
\definecolor{darktangerine}{rgb}{1.0, 0.66, 0.07}
\definecolor{deepsaffron}{rgb}{1.0, 0.6, 0.2}

\newcommand{\bv}[1]{{\color{blue-violet}{#1}}}

\usepackage{soul}
\usepackage{ulem}

\usepackage{tikzsymbols}
\usepackage{geometry}
\newcommand{\gka}{\Gamma{\text{-KA}}}

\newcommand{\ams}{\mathrm{ZMS}_}
\newcommand{\agms}{\mathrm{ZMS}_{\Gamma}}

\newcommand{\zms}{\mathrm{ZMS}_}

\newcommand{\Spec}{\mathrm{Spec}}
\newcommand{\ms}{\mathrm{MS}_}
\newcommand{\ims}{\mathrm{MS}}
\newcommand{\gms}{\mathrm{MS}_{\Gamma}}

\newtheorem{theorem}{Theorem}[section]

\newtheorem{lemma}[theorem]{Lemma}
\newtheorem*{lemma*}{Lemma}

\newtheorem{oprb}[theorem]{Open Problem}
\newtheorem{obs}[theorem]{Observation}

\theoremstyle{definition}

\theoremstyle{definition}

\newtheorem*{prb*}{Open Problem}

\newtheorem{exm}[theorem]{Example}

\def\gr{\mathcal{G}}

\vspace{5cm}

\begin{document}


\title
{\Large \sc \bf {Note on zero-sum magic  squares on Abelian groups}
}
\date{}
\author{{{Sylwia Cichacz$^{1}$, Dalibor Froncek$^{2}$}}\\
\normalsize $^1$AGH University of Krakow, Poland, cichacz@agh.edu.pl\\
\normalsize $^2$University of Minnesota Duluth, U.S.A.,  dalibor@d.umn.edu
}

\maketitle

\begin{abstract}
Let $(\Gamma,+)$ be an Abelian group of order $n^2$. A $\Gamma$-magic square of order $n$ is an $n\times n$ array whose entries are pairwise distinct elements of $\Gamma$ such that all row sums, column sums, and the two main diagonal sums are equal to the same element $\mu \in \Gamma$, called the magic constant.

A combinatorial design is called $\Gamma$-additive if its point set is a subset of an Abelian group $\Gamma$ and every block has sum zero. If the point set coincides with $\Gamma$, the design is said to be strictly $\Gamma$-additive. Motivated by this notion, {we construct $\Gamma$-magic squares with magic constant $\mu=0$ whose rows, columns, and two main diagonals can be used as blocks of a strictly $\Gamma$-additive design. We call such a square {zero-sum $\Gamma$-magic square}.}

In this paper, we establish necessary and sufficient conditions for the existence of zero-sum $\Gamma$-magic squares.
\end{abstract}



\section{Introduction}\label{sec:intro}

Magic squares are among the oldest and most studied combinatorial objects. A \emph{magic square of order $n$},
denoted by $\ims(n)$, is an $n\times n$ array whose entries are the integers $1,2,\dots,n^2$ and such that the sum
of the entries in each row, each column, and the two main diagonals is equal to the same value, called the
\emph{magic constant}, namely $c=n(n^2+1)/2$.When $n=1$, the magic square is \emph{trivial}.
We state the theorem here for completeness. For an overview, see, e.g., Chapter 6 Section 33 in~\cite{handbook}.

\begin{theorem}[\cite{handbook}] \label{thm:msquare}
	There exists a  {non-trivial }{magic square} of side $n$ if and only if $n>2$.
\end{theorem}

A natural generalization of magic squares is given by \emph{magic rectangles}. An $m\times n$ magic rectangle
is an array with entries $1,2,\dots,mn$ such that all row sums are equal to the \emph{row constant} $r=n(mn+1)/2$ {and} all column sums are equal to the \emph{column constant} $c=m(mn+1)/2$. 

Magic squares and rectangles admit many further generalizations. One may impose restrictions on the entries, 
for instance, requiring them to be the elements
of an Abelian group \cite{CicJaco,Cichacz-Hincz-2,CicFro,Evans,Sun-Yihui}. Let $\Gamma$ be an Abelian group of order $n^2$. A \emph{$\Gamma$-magic square},
denoted by $\gms(n)$, is an $n\times n$ array whose entries are the elements of $\Gamma$, each appearing
exactly once, such that all row, column, main diagonal and antidiagonal sums are equal to the same element  $\mu\in\Gamma$,
called the \emph{magic constant}.

To avoid confusion between the order of the group and the size of the array, we shall refer to $n$ as the
\emph{side} of the square. Thus, $\ms{\Gamma}(n)$ denotes a magic square of side~$n$.  It was proved recently the following.
\begin{theorem}[Cichacz, Froncek \cite{CicFro}]\label{thm:main}
Let $\Gamma$ be an Abelian group of order $n^2$. There exists a non-trivial $\gms(n)$  if and only if $n>2$.
\end{theorem}

The identity element {of an Abelian group} $\Gamma$ will be denoted by $0$. A subset $S$ of $\Gamma$ is called a \emph{zero-sum subset} if $\sum_{g\in S} g = 0$.
Zero-sum problems play a significant role in several areas of combinatorics
\cite{BurattiMerolaNakic2025,CicZ,CicJaco,CicSur,CaggegiFalconePavone2017,ErdGinZiv,Skolem57}.
One notable example is provided by \emph{additive combinatorial designs}.
According to~\cite{CaggegiFalconePavone2017} and its recent extension
in~\cite{BurattiMerolaNakic2025}, a combinatorial design is said to be
\emph{$\Gamma$-additive} if its point set is a subset of an Abelian group $\Gamma$
and each block is zero-sum. If the point set coincides with $\Gamma$, the design
is called \emph{strictly $\Gamma$-additive}.
Therefore, if a $\ms{\Gamma}(n)$ has the magic constant $\mu=0$, then it naturally defines a strictly $\Gamma$-additive
partial linear space whose point set is $\Gamma$ and whose blocks are given by the rows, columns,
main diagonals and antidiagonals of the square.
In this sense, a $\Gamma$-magic square with zero magic constant can be viewed as a strictly
$\Gamma$-additive combinatorial structure.

Hence, a $\ms{\Gamma}(n)$ with the magic constant $\mu=0$ will be called a \emph{zero-sum $\ms{\Gamma}(n)$} and denoted by $\ams{\Gamma}(n)$.

While nontrivial $\Gamma$-magic squares exist for all $n>2$, the zero-sum requirement $\mu=0$ is far from automatic when $n$ is even. In particular, the structure of involutions in $\Gamma$ plays a crucial role. Recall that an element $\iota\in\Gamma$ of order 2 (i.e., $\iota\neq 0$ and $2\iota=0$) is called an \emph{involution}. 
For convenience, let $\gr$ denote the set consisting of all Abelian groups that are
of odd order (and thus have no involution) or contain more than one involution.

In this paper we give sufficient and necessary conditions for existence of a $\agms(n)$.
\begin{restatable}{theorem}{CoPthm}
\label{thm:mainams}
Let $\Gamma$ be an Abelian group of order $n^2>1$. There exists a non-trivial $\agms(n)$  if and only if $n>2$ and $\Gamma\in\gr$.
\end{restatable}
	The paper is organized as follows.  Section~\ref{sec:defs and related}  contains preliminary lemmas and observations needed to build zero-sum $\Gamma$-magic squares. In Section~\ref{constructions}, we present the main constructions.  We conclude the paper with some final remarks and open questions in Section~\ref{sec:conclusion}.

 
\section{Related results}\label{sec:defs and related}

 In~\cite{CicZ}, Cichacz introduced concept of a $\Gamma$-\textit{Kotzig array}. 	A \emph{$\Gamma$-Kotzig array} $\gka(j,k)$ where $\Gamma$ is an Abelian group of order $k$  is
	a $j\times k$ grid where each row is a permutation of elements of $\Gamma$ and each column has the same sum.    By adding to all elements in each row the inverse of its first entry we get all  elements in the first column equal  $0$ and therefore  all  column sums will equal to $0$. 

It was proved the following.
\begin{theorem}[Cichacz \cite{CicZ}]\label{thm:Kotzig}
A $\Gamma$-Kotzig array of size $j \times k$ exists if and only if
$j>1$ and $j$ is even or $\Gamma\in \gr$. Moreover, whenever a $\Gamma$-Kotzig array exists, then there exists one with the column sums equal to $0$.
\end{theorem}

The following two lemmas play an important role in our reduction of the problem.

\begin{lemma}[Cichacz, Froncek~\cite{CicFro}]
	Let $\Gamma$ be an Abelian group of order $n^2$. 
	Let  $\Gamma\cong \Gamma_0\oplus H$ for some groups $\Gamma_0$ with $|\Gamma_0|=m^2$, $m>1$ and $H$ with $|H|=k^2$. If there exists a $\Gamma_0$-magic square   $\ms{\Gamma_0}(m)$   with the  magic sum $\delta$ and {an $H$-Kotzig array of size $m\times k^2$}, then there exists a $\Gamma$-magic square $\ms{\Gamma}(n)$ with the magic sum $(k\delta,0)$.\label{lem:SWL}
\end{lemma}

\begin{lemma}[Cichacz, Froncek~\cite{CicFro}]\label{lem:DWL}
	Let $\Gamma$ be an Abelian group of order $n^2$ such that  $\Gamma\cong Z_{k^2m_1}\oplus Z_{m_2}$ for  $k>2$. If there exists a $Z_{m_1}\oplus Z_{m_2}$-magic square   $\ms{Z_{m_1}\oplus Z_{m_2}}(n/k)$   with the  magic sum $\delta=(\delta_1,\delta_2)$, then there exists a $\Gamma$-magic square $\ms{\Gamma}(n)$ with the magic sum  $\mu=(k^3\delta_1+n(k^2-1)/2,k\delta_2)$.
\end{lemma}

Observe that by Lemma~\ref{lem:SWL} we immediately obtain the following.
\begin{lemma}\label{lem:SWL2}
	Let $\Gamma$ be an Abelian group of order $n^2$. 
	Let  $\Gamma\cong \Gamma_0\oplus H$ for some groups $\Gamma_0$ with $|\Gamma_0|=m^2$, $m>1$ and $H$ with $|H|=k^2$. If there exists a   $\ams{\Gamma_0}(m)$ and {an $H$-Kotzig array of size $m\times k^2$}, then there exists a  $\ams{\Gamma}(n)$.
\end{lemma}

\section{Constructions}\label{constructions}

When needed, we will denote the row, column, main diagonal, and backward main diagonal sums by $\rho_i,\sigma_j,\tau$, and $\kappa$, respectively, possibly with superscripts when we need to distinguish between different squares.

\subsection{Construction for $n$ odd}\label{sec: n odd}

We start with  a useful lemma.

\begin{lemma}\label{primep}
	Let  $\Gamma$ be an Abelian group of order $n^2$ where $n$ is an odd integer. If $\Gamma\cong Z_{n^2}$ or $\Gamma\cong Z_n\oplus Z_n$, then there exists a $\ams{\Gamma}(n)$.
\end{lemma}
\begin{proof} 
	Assume first that $\Gamma\cong Z_{n^2}$. It follows from Theorem~\ref{thm:msquare} that there exists a 
	magic square $A'=\{a_{i,j}'\}_{i,j=1}^{n}=\ms{}(n)$ (in integers). Replacing only $n^2$ by 0 and performing addition modulo $n^2$ gives us an $A=\gms(n)=\{a_{i,j}\}_{i,j=1}^{n}$ with the   magic constant $\delta=n(n^2-1)/2$. Observe that $(n^2-1)/2\in Z_{n^2}$ since $n$ is odd. Let now  $B=\{b_{i,j}\}_{i,j=1}^{n}$ be defined as $b_{i,j}=a_{i,j}-(n^2-1)/2$ for $i,j\in\{1,2,\ldots,n\}$. Obviously $B$ is a $\Gamma$-magic square with the magic constant $\mu=\delta-n(n^2-1)/2=0$.
	
	Let now $\Gamma\cong Z_n\oplus Z_n$, then we will use the construction presented in~\cite{Sun-Yihui}. Let $B=\{b_{i,j}\}_{i,j=1}^{n}$ be defined as $b_{i,j}=(i-1,j-1)$ for $i,j\in\{1,2,\ldots,n\}$. Observe that the row sums are $\sum_{j=1}^{n}(i-1,j-1)=(n(i-1),n(n-1)/2)=(0,0)$, the column sum is $\sum_{i=1}^{n}(i-1,j-1)=(n(n-1)/2,n(j-1))=(0,0)$, the main diagonal sum is $\sum_{i=1}^{n}(i-1,i-1)=(n(n-1)/2,n(n-1)/2)=(0,0)$ and the main antidiagonal sum is $\sum_{i=1}^{n}(i-1,n-i)=(n(n-1)/2,n-n(n-1)/2)=(0,0)$.
\end{proof}

 We will state now our main result in this section.

\begin{theorem}\label{thm:main-odd}
	Let $\Gamma$ be an Abelian group of order $n^2$ where $n\geq3$ is odd. Then there exists a  $\agms(n)$ of side $n$.
\end{theorem}

\begin{proof}	
	By Fundamental Theorem of Finite Abelian Groups the group $\Gamma$ is either  cyclic and the result follows directly from Lemma~\ref{primep}, or $\Gamma\cong H\oplus K$, where $|H|=p^{2\lambda}$ for some prime number $p>2$ and integer $\lambda\ge 1$. In the latter case, we either have $H\cong Z_{p^{2\lambda}}$ and $\ams{H}(p^{\lambda})$ exists by Lemma~\ref{primep}, or $H\cong Z_{p^{\alpha}}\oplus Z_{p^{\beta}}$  for $\alpha+\beta=2\lambda$. In this case without loss of generality we can assume that $\beta\leq \alpha$. By Lemma~\ref{primep} there exists a $\ams{Z_{p^{\beta}}\oplus Z_{p^{\beta}}}(p^{\beta})$. Applying now Lemma~\ref{lem:DWL} there exists an $\ms{Z_{p^{\alpha}}\oplus Z_{p^{\beta}}}(p^{\lambda})=A=\{a_{i,j}\}_{i,j=1}^{p^{\lambda}}$ with the magic constant 
	$\delta=(p^{\lambda}(p^{\alpha-\beta}-1)/2,0)$. Note that $(p^{\alpha-\beta}-1)/2 \in Z_{p^{\alpha}}$. Let now  $B=\{b_{i,j}\}_{i,j=1}^{p^{\lambda}}$ be defined as $b_{i,j}=a_{i,j}-((p^{\alpha-\beta}-1)/2,0)$ for $i,j\in\{1,2,\ldots,p^{\lambda}\}$. Obviously $H$ is a $\Gamma$-magic square with the magic constant $\mu=\delta-(p^{\lambda}(p^{\alpha-\beta}-1)/2,0)=(0,0)$.

	 If 
{$\Gamma=H$} 
 	we are done, otherwise 
{$|K|>1$ is odd} and
	$K \in \mathcal{G}$. Then by Theorem~\ref{thm:Kotzig} there exists a {$K$-Kotzig array}  of size $p^{\lambda}\times |K|$ and  we apply Lemma~\ref{lem:SWL2}.

\end{proof}


\subsection{Construction for $n=2^s$}\label{sec:n=2tos}
The main idea of the proofs in this section  is a construction based on block decomposition. Roughly speaking, for $\beta > \alpha$, we construct a zero‑sum magic square of side $2^{\beta}$ from a zero‑sum magic square of side $2^{\alpha}$ by partitioning the larger square into suitably chosen subsquares whose row, column, and diagonal sums cancel out. We begin with the following lemma, which plays a crucial role in our constructions.

\begin{lemma}\label{lem: Z_2^c + Z_2^c+d}
Let $\Gamma=Z_{2^{\alpha}}\oplus H$ be an Abelian group of order $2^{2\gamma}$ for $\gamma\geq 3$ and $2\gamma=\alpha+\beta$.  Let $\Gamma_0=Z_{2^{\alpha-2}}\oplus H$.   If $\alpha\leq \beta$ and there exists a  $\ams{\Gamma_0}(2^{\gamma-1})$, then there exists a $\agms(2^{\gamma})$.
\end{lemma}

\begin{proof} 
	Note that $|H|=2^{\beta}$. We will build an $M_{\gamma}=\ams{\Gamma}(2^{\gamma})$, out of four subsquares $M^{s,t}\{(a,b)_{i,j}^{s,t}\}_{i,j=1}^{2^{\gamma-1}}, s,t\in[1,2]$. 
	
By the assumption of the theorem a $\ams{\Gamma_0}(2^{\gamma-1})=M_{\gamma-1}=\{(a,b)_{i,j}\}_{i,j=1}^{2^{\gamma-1}}$ exists.
	First we construct a $2^{\gamma-1}\times2^{\gamma-1}$ square $M'_{\gamma-1}$   from $M_{\gamma-1}$ by replacing each
	$m_{i,j}=(a,b)_{i,j}\in M_{\gamma-1}$ with
	\begin{equation*}
		m'_{i,j} = ({4}a,b)_{i,j}.
	\end{equation*}
	Because $M_{\gamma-1}$ is a zero-sum square, it should be obvious that all row, column, and relevant diagonal sums in $M'_{\gamma-1}$ are also zero.

	Then we construct the subsquares $M^{s,t}$ by combining $M'_{\gamma-1}$ with four {residual squares} $R^{s,t}=\{(c,d)_{i,j}^{s,t}\}_{i,j=1}^{2^{\gamma-1}}, s,t\in[1,2]$. For every $i,j\in[1,2^{\gamma-1}]$ we define
	\begin{align*}
		r^{1,1}_{i,j}	&= (0,0),	
		&r^{1,2}_{i,j}	&=
		\begin{cases*}
			(1,0) \text{\ for\ } i+j \text{\ even}	\\
			(3,0) \text{\ for\ } i+j \text{\ odd}	
		\end{cases*},\\ 
		r^{2,1}_{i,j} 	&=
		\begin{cases*}
			(3,0) \text{\ for\ }  i+j \text{\ even}		\\
			(1,0) \text{\ for\ }  i+j \text{\ odd}		
		\end{cases*},
		&r^{2,2}_{i,j} 	&= (2,0)
	\end{align*}

	First we recall that $\alpha\leq \beta$ and therefore
	$
		\gamma = (\alpha + \beta)/2  \geq \alpha
	$
	and observe that all relevant sums (except for back diagonals in $R^{1,2}$ and $R^{2,1}$) are also zero. For instance, 
	we have
	\begin{equation*}\label{eq:residual sum}
		\rho^{1,2}_1 
			=(1,0)+(3,0)+\dots+(1,0)+(3,0)
			= 2^{\gamma-2}(1+3,0) 
			= (2^{\gamma},0) = (0,0)
	\end{equation*}
	because $2^{\gamma}\in Z_{2^{\alpha}}$ and $\gamma\geq \alpha$. Similar calculations give all row and column sums as well as the main diagonals in $R^{1,1}$ and $R^{2,2}$.
	For the back diagonals in $R^{1,2}$ and $R^{2,1}$ we have
	$$
		\kappa^{1,2}=2^{\gamma-1}(3,0) \text{\ and\ } \kappa^{2,1}=2^{\gamma-1}(1,0)
	$$
	and their sum across the backward diagonal in the complete 
	$2^{\gamma}\times2^{\gamma}$ rectangle is  $2^{\gamma-1}(3+1,0) =(2^{\gamma+1},0)=(0,0)$.
	Now we just set for every subsquare $M^{s,t}$
	$$
		m^{s,t}_{i,j}=m'_{i,j} + r^{s,t}_{i,j}.
	$$
	The entries in $M_{\gamma}$ are exactly all elements of 
	$Z_{{2^{\alpha}}}\oplus Z_{2^{\beta}}$
	and because all relevant sums in $M'_{\gamma-1}$ and $R^{s,t}$ are equal to $(0,0)$, the proof is complete.
\end{proof}

\begin{figure}[H]
$$
\begin{array}{||c|c|c|c||c|c|c|c||}
\hline
\hline
	 (0,2)  & (0,0) & (4,2) & (4,4)	&(1,2)  & (3,0) & (5,2) & (7,4)\\ \hline
	 (4,6)  & (4,0) & (0,6) & (0,4)	&(7,6)  & (5,0) & (3,6) & (1,4)\\ \hline
	 (4,1)  & (4,7) & (4,3) & (4,5)	&(5,1)  & (7,7) & (5,3) & (7,5)\\ \hline
	 (0,7)  & (0,1) & (0,5) & (0,3)	&(3,7)  & (1,1) & (3,5) & (1,3)\\ \hline
\hline
	 (3,2)  & (1,0) & (7,2) & (5,4)	&(2,2)  & (2,0) & (6,2) & (6,4)\\ \hline
	 (5,6)  & (7,0) & (1,6) & (3,4)	&(6,6)  & (6,0) & (2,6) & (2,4)\\ \hline
	 (7,1)  & (5,7) & (7,3) & (5,5)	&(6,1)  & (6,7) & (6,3) & (6,5)\\ \hline
	 (1,7)  & (3,1) & (1,5) & (3,3)	&(2,7)  & (2,1) & (2,5) & (2,3)\\ \hline
\hline	
\end{array}
$$
\caption{$\ams{Z_{8}\oplus Z_{8}}(8)$}	
 	\label{fig:example Z_8 x Z_8}
\end{figure}

\begin{exm}\label{exm: Z_8 x Z_8}
	In Figure~\ref{fig:example Z_8 x Z_8} we show the construction of $\ams{Z_{8}\oplus Z_{8}}(8)$ from $\ams{Z_{2}\oplus Z_{8}}(4)$, which is given in Figure~\ref{fig: Z_2 x Z_8}.
\end{exm}

	\begin{figure}[H]
	\begin{center}
	\begin{tabular}{|c|c|c|c|}
	\hline
	    (0,2)  & (0,0) & (1,2) & (1,4)\\ \hline
	    (1,6)  & (1,0) & (0,6) & (0,4)\\ \hline
	    (1,1)  & (1,7) & (1,3) & (1,5)\\ \hline
	    (0,7)  & (0,1) & (0,5) & (0,3)\\ \hline
	  \end{tabular}
	  \caption{$\ams{Z_{2}\oplus Z_8}(4)$}
	  \label{fig: Z_2 x Z_8}
	\end{center}
	\end{figure}

We now treat two special cases, in which the group $\Gamma$ contains a small component. 

\begin{lemma}\label{obs: Z_2 + Z_2^c}
 	Let $\Gamma= Z_{2}\oplus Z_{2^{2\alpha-1}}$ and $\alpha\geq2$. Then there exists a $\agms(2^\alpha)$.
\end{lemma}

\begin{proof} 
	We proceed by induction on $\alpha$. As in the proof of Lemma~\ref{lem: Z_2^c + Z_2^c+d} we enlarge a zero-sum square of side  $2^\alpha$ to one of side $2^{\alpha+1}$ using block decomposition. The base case for $\alpha=2$ is shown in Figure~\ref{fig: Z_2 x Z_8}.

	Let $M_\alpha=\{(a,b)_{i,j}\}_{i,j=1}^{2^\alpha}$ be a $\ams{Z_{2}\oplus Z_{2^{2\alpha-1}}}(2^\alpha)$, which exists by inductive hypothesis. We build an $M_{\alpha+1}=\ams{Z_{2}\oplus Z_{2^{2\alpha+1}}}(2^{\alpha+1})$, out of four subsquares $M^{s,t}\{(a,b)_{i,j}^{s,t}\}_{i,j=1}^{2^\alpha}, s,t\in[1,2]$. The subsquare $M^{1,1}$ is obtained from $M_\alpha$ by
	\begin{equation*}\label{eq:M^11}
		m^{1,1}_{i,j} = (a,4b)_{i,j}
	\end{equation*}
	for every $(a,b)_{i,j}\in M_\alpha$.
	The remaining subsquares $M^{s,t}$ are built from $2\times2$ blocks $B_k$ where $k\in[1,2^{2\alpha}-1]$ and $k\not\equiv0\pmod4$, defined as

	\begin{equation*}
			B_k=
		\begin{array}{|c|c|}
			\hline
			(0,k)				&(0,-k)\\
			\hline
			(1,-k)	&(1,k)	\\
			\hline
		\end{array}	
	\end{equation*}	
	with row sums $(0,0)$, column sums $(1,0)$ and diagonal sums $(1,2k)$ and $(1,-2k)$, respectively. There are $3\cdot2^{2\alpha-2}$ such squares and they are well-defined, because when $k\equiv1\pmod4$, then $-k\equiv3\pmod4$ and vice versa, and when $k\equiv2\pmod8$, then $-k\equiv6\pmod8$ and vice versa.
	
	We then place the blocks in the subsquares so that when $B_k$ is on the main or backward diagonal, then also $B_{2^{2\alpha}-k}$ is on the same diagonal. Notice that the diagonal sums in $B_{2^{2\alpha}-k}$ are
	\begin{equation*}\label{eq:main diag sums}
		(0,2^{2\alpha}-k) + (1,2^{2\alpha}-k) = (1,-2k)
	\end{equation*}
	for the main diagonal and
	\begin{equation*}\label{eq:back diag sums}
		(0,2^{2\alpha}+k) + (1,2^{2\alpha}+k) = (1,2k)
	\end{equation*}
	for the backward diagonal. Hence, the sum of the block diagonals of $B_{k}$ and $B_{2^{2\alpha}-k}$ is always $(0,0)$.
	
	The subsquares $M^{s,t}$ contain even number of blocks, which implies that also the row and column sums in each subsquare are $(0,0)$.
\end{proof}

\begin{exm}\label{exm: Z_2 x Z_32}
	In Figure~\ref{fig:example Z_2 x Z_32} we show the construction of $\ams{Z_{2}\oplus Z_{32}}(8)$.
\end{exm}

\begin{figure}[H]
$$
\begin{array}{||c|c|c|c||c|c|c|c||}
\hline
\hline
(0,8)  & (0,0) & (1,8) 	& (1,16)		&(0,1)  & (0,31)& (0,5)  & (0,27)\\ \hline
(1,24) & (1,0) & (0,24)	& (0,16)		&(1,31) & (1,1) & (1,27) & (1,5)\\ \hline
(1,4)  & (1,28)& (1,12)	& (1,20)		&(0,11) & (0,21)& (0,15) & (0,17)\\ \hline
(0,28) & (0,4) & (0,20)	& (0,12)		&(1,21) & (1,11)& (1,17) & (1,15)\\ \hline
\hline
(0,3)  & (0,29)& (0,7) 	& (0,25)		&(0,2)  & (0,30)& (0,6)  & (0,26)\\ \hline
(1,29) & (1,3) & (1,25) & (1,7)			&(1,30) & (1,2) & (1,26) & (1,6)\\ \hline
(0,9)  & (0,23)& (0,13) & (0,19)		&(0,10) & (0,22)& (0,14) & (0,18)\\ \hline
(1,23) & (1,9) & (1,19) & (1,13)		&(1,22) & (1,10)& (1,18) & (1,14)\\ \hline
\hline
\end{array}
$$
\caption{$\ams{Z_{2}\oplus Z_{32}}(8)$}	
 	\label{fig:example Z_2 x Z_32}
 \end{figure}


\begin{lemma}\label{obs: Z_2 + Z_2+ Z_2^c}
 	Let $\Gamma= H \oplus Z_{2^{2\alpha-2}}$ for $|H|=4$ and $\alpha\geq2$. Then there exists a $\agms(2^\alpha)$.
\end{lemma}

\begin{proof} 
	Observe that $H=Z_2\oplus Z_2$ or $H=Z_4$. We proceed by induction on $\alpha$.  The base cases for $\alpha=2$ are shown in Figure~\ref{fig: Z_2 x Z_2 x Z_4 1}.
\begin{figure}[H]
\begin{center}
\begin{tabular}{|c|c|c|c|}
\hline
  (0,0,0)&(0,0,1)&(1,0,0)&(1,0,3)\\\hline
  (0,0,3)&(0,0,2)&(1,0,1)&(1,0,2)\\\hline
(1,1,0)&(1,1,3)&(0,1,0)&(0,1,1)\\\hline
(1,1,1)&(1,1,2)&(0,1,3)&(0,1,2)\\\hline
\end{tabular} ~~~~~~~~\begin{tabular}{|c|c|c|c|}
\hline
  (2,0)&(0,1)&(3,0)&(3,3)\\\hline
  (2,3)&(0,2)&(1,1)&(1,2)\\\hline
(1,0)&(1,3)&(0,0)&(2,1)\\\hline
(3,1)&(3,2)&(0,3)&(2,2)\\\hline

  \end{tabular} 
\end{center}
\caption{$\ams{Z_2\oplus Z_2\oplus Z_4}(4)$ and $\ams{Z_4\oplus Z_4}(4)$ }
\label{fig: Z_2 x Z_2 x Z_4 1}
\end{figure}
 
	Now let $M_\alpha=\{(a,b)_{i,j}\}_{i,j=1}^{2^\alpha}$ be a $\ams{H\oplus  Z_{2^{2\alpha-2}}}(2^\alpha)$, which exists by inductive hypothesis. We build an $M_{\alpha+1}=\ams{H\oplus Z_{2^{2\alpha}}}(2^{\alpha+1})$ out of four subsquares $M^{s,t}\{(a,b)_{i,j}^{s,t}\}_{i,j=1}^{2^\alpha}, s,t\in[1,2]$. The subsquare $M^{1,1}$ is obtained from $M_\alpha$ by
	\begin{equation*}\label{eq:M^11}
		m^{1,1}_{i,j} = (a,4b)_{i,j}
	\end{equation*}
	for every $(a,b)_{i,j}\in M_\alpha$. 
	
Let $H=\{0,a_1,a_2,a_3\}$. Without loss of generality, we may assume that  $2a_1=0$ and $a_3=a_1+a_2$.  The remaining subsquares $M^{s,t}$ are built from $4\times4$ blocks $B_k$ where $k\in[1,\ldots,2^{2\alpha-2}-1]$ and $k\not\equiv0\pmod4$, defined as follows.

	\begin{equation*}
			B_k=
	\begin{array}{|c|c|c|c|}
\hline
  (0,k) &(0,-k) &(a_1,k)&(a_1,-k)\\\hline
  (0,2^{2\alpha-1}+k)   &(0,2^{2\alpha-1}-k)   &(a_1,2^{2\alpha-1}+k)&(a_1,2^{2\alpha-1}-k)\\\hline
 (a_2,-k)&(a_2,k)  &(a_3,-k)&(a_3,k)\\\hline
  (-a_2,2^{2\alpha-1}-k)  &(-a_2,2^{2\alpha-1}+k)    &(-a_3,2^{2\alpha-1}-k)&(-a_3,2^{2\alpha-1}+k)\\\hline
 
\end{array} 
	\end{equation*}	
	with row sums, column sums  and diagonal sums equal $(0,0)$, since either $2a_2=2a_3=0$ when $H=Z_2\oplus Z_2$, or $a_2+a_3=0$ when $H=Z_4$.  This finishes the proof.
\end{proof}

\begin{exm}\label{exm: Z_2 x Z_32}
	In Figure~\ref{fig:example Z_2 x Z_2 x Z_16} we show the construction of $\ams{Z_{2}\oplus Z_2 \oplus Z_{16}}(8)$.
\end{exm}

\begin{figure}[H]
$$
\begin{array}{||c|c|c|c||c|c|c|c||}
\hline
\hline
  (0,0,0)&(0,0,4)&(1,0,0)&(1,0,12)&(0,0,1)   &(0,0,15)   &(1,0,1)&(1,0,15)\\\hline
  (0,0,12)&(0,0,8)&(1,0,4)&(1,0,8)&(0,0,9)   &(0,0,7)   &(1,0,9)&(1,0,7)\\\hline
(1,1,0)&(1,1,12)&(0,1,0)&(0,1,4)& (1,1,15)&(1,1,1)  &(0,1,15)&(0,1,1)\\\hline
(1,1,4)&(1,1,8)&(0,1,12)&(0,1,8)&(1,1,7)&(1,1,9)  &(0,1,7)&(0,1,9)\\\hline\hline
  (0,0,2) &(0,0,14) &(1,0,2)&(1,0,14)&(0,0,3) &(0,0,13) &(1,0,3)&(1,0,13)\\\hline
  (0,0,10)   &(0,0,6)   &(1,0,10)&(1,0,6)&(0,0,11)   &(0,0,5)   &(1,0,11)&(1,0,5)\\\hline
 (1,1,14)&(1,1,2)  &(0,1,14)&(0,1,2)&(1,1,13)&(1,1,3)  &(0,1,13)&(0,1,3)\\\hline
  (1,1,6)  &(1,1,10)    &(0,1,6)&(0,1,10)& (1,1,5)  &(1,1,11)    &(0,1,5)&(0,1,11)\\\hline
\hline
\end{array}
$$
\caption{$\ams{Z_{2}\oplus Z_2 \oplus Z_{16}}(8)$}	
 	\label{fig:example Z_2 x Z_2 x Z_16}
 \end{figure}


\begin{theorem}\label{thm:order 2^s}
	Let $|\Gamma| = 2^{2\alpha}$, $\Gamma\in \gr$ and $\alpha\geq2$. Then there exists a {$\agms(2^{\alpha})$}.
\end{theorem}

\begin{proof}
	The proof will be by induction on $|\Gamma|$. Let {$\alpha=2$}, then the $\ams{\Gamma}(4)$ for $\Gamma\in\{Z_2\oplus Z_8,Z_4\oplus Z_4,Z_2\oplus Z_2\oplus Z_4,Z_2\oplus Z_2\oplus Z_2\oplus Z_2\}$ are shown in Figures~\ref{fig: Z_2 x Z_8 and Z_4 x Z_4}--\ref{fig: Z_2^4}.

	Thus assume that $\alpha\geq 3$. Let $\Gamma=Z_{2^{\beta_1}}\oplus Z_{2^{\beta_2}}\oplus\dots\oplus Z_{2^{\beta_t}}$ and $\beta_1\leq\beta_2\leq\dots\leq\beta_t$. Since $\Gamma\in\gr$, there is $t>1$.

	When $\beta_1=\beta_2=1$ or $\beta_1=2$  we have {$\Gamma=H\oplus K$} for $|H|=4$.  	Because $\alpha\geq3$, we have {$|K|\geq 2^{4}$}.  If  now $K\in \gr$, then there exists a $\ams{K}(2^{\alpha-1})$ by inductive hypothesis and an $H$-Kotzig array of size $2^{\alpha-1}\times 4$ based on Theorem~\ref{thm:Kotzig}. The existence of $\agms(2^\alpha)$ then follows from Lemma~\ref{lem:SWL2}. 
	Suppose that $K\notin \gr$, then $\Gamma={Z_{2}\oplus Z_{2}}\oplus Z_{2^{2\alpha-2}}$ or $\Gamma={Z_{4}}\oplus Z_{2^{2\alpha-2}}$  for $\alpha\geq{3}$. Thus we apply Lemma~\ref{obs: Z_2 + Z_2+ Z_2^c}. 
 
	Suppose now that $\beta_1=1$ and $\beta_2>1$.  Then there exists an index $i$ such that $\beta_{i}$ is also odd and greater than one.  Therefore  $\Gamma=H\oplus K$ where $H=Z_{2}\oplus Z_{2^{\beta_i}}$ and {$|H|=\beta_i+1>2$ is even}. 
	The existence of a zero-sum $H$-magic square follows from Lemma~\ref{obs: Z_2 + Z_2^c}.   If $t=2$, then $|K|=1$ and the claim follows immediately. For $t>1$ there is $|K|$ even and  there exists a $K$-Kotzig array of size $2^{(\beta_1+\beta_i)/2}\times |K|$ by Theorem~\ref{thm:Kotzig}. The existence of $\agms(2^\alpha)$ again follows from Lemma~\ref{lem:SWL2}.

	Finally we can assume that $\beta_1\geq 3$. Let $H=Z_{2^{\beta_2}}\oplus Z_{2^{\beta_3}}\oplus\dots\oplus Z_{2^{\beta_t}}$ and $\lambda=\sum_{i=2}^t \beta_i$. Denote $\Gamma_0=Z_{2^{\beta_1-2}}\oplus H$ and observe that $\Gamma_0\in \gr$ and $\lambda\geq \beta_1$.  By inductive hypothesis there exists a 
	$\ams{\Gamma_0}(2^{\alpha-1})$   and by Lemma~\ref{lem: Z_2^c + Z_2^c+d} there exists a 
	$\ams{\Gamma}(2^{\alpha})$.  This finishes the proof.
\end{proof}

\begin{figure}[H]
\begin{subfigure}[t]{0.5\textwidth}
\begin{center}
\begin{tabular}{|c|c|c|c|}
\hline
  (0,2)  & (0,0) & (1,2) & (1,4)\\ \hline
    (1,6)  & (1,0) & (0,6) & (0,4)\\ \hline
    (1,1)  & (1,7) & (1,3) & (1,5)\\ \hline
    (0,7)  & (0,1) & (0,5) & (0,3)\\ \hline
  \end{tabular}
\subcaption{$\ams{Z_2\oplus Z_8}$}
\end{center}
\end{subfigure}
\hfill
\begin{subfigure}[t]{0.5\textwidth}
\begin{center}
\begin{tabular}{|c|c|c|c|}
\hline
  (2,0)&(0,1)&(3,0)&(3,3)\\\hline
  (2,3)&(0,2)&(1,1)&(1,2)\\\hline
(1,0)&(1,3)&(0,0)&(2,1)\\\hline
(3,1)&(3,2)&(0,3)&(2,2)\\\hline

  \end{tabular} 
\subcaption{$\ams{Z_4\oplus Z_4}(4)$}
\end{center}
\end{subfigure}
\caption{}
\label{fig: Z_2 x Z_8 and Z_4 x Z_4}
\end{figure}

\vskip10pt

\begin{figure}[H]
\begin{center}
\begin{tabular}{|c|c|c|c|}
\hline
  (0,0,0)&(0,0,1)&(1,0,0)&(1,0,3)\\\hline
  (0,0,3)&(0,0,2)&(1,0,1)&(1,0,2)\\\hline
(1,1,0)&(1,1,3)&(0,1,0)&(0,1,1)\\\hline
(1,1,1)&(1,1,2)&(0,1,3)&(0,1,2)\\\hline
\end{tabular} 
\end{center}
\caption{$\ams{Z_2\oplus Z_2\oplus Z_4}(4)$}
\label{fig: Z_2 x Z_2 x Z_4}
\end{figure}

\vskip10pt
\begin{figure}[H]
\begin{center}
\begin{tabular}{|c|c|c|c|}
\hline
    (0,0,0,0)  & (0,1,0,0) & (0,0,0,1) & (0,1,0,1)\\ \hline
    (1,1,0,0)  & (1,0,0,0) & (1,1,0,1) & (1,0,0,1)\\ \hline
    (0,0,1,0)  & (0,1,1,0) & (0,0,1,1) & (0,1,1,1)\\ \hline
    (1,1,1,0)  & (1,0,1,0) & (1,1,1,1) & (1,0,1,1)\\ \hline
\end{tabular} 
\end{center}
\caption{$\ams{Z_2\oplus Z_2\oplus Z_2\oplus Z_2}(4)$}
\label{fig: Z_2^4}
\end{figure}


\subsection{Main result} \label{sec:main result}

We start with an easy observation. 
\begin{obs}\label{involution}Let $\Gamma$ be an Abelian group of order $n^2$. If $\Gamma \notin \gr$, then there does not exist a $\agms(n)$.\label{gr}
\end{obs}
\begin{proof} Note that $n$ is even since $\Gamma \notin \gr$, and moreover there exists exactly one involution $\iota\in \Gamma$ and thus $\sum_{g\in \Gamma}g=\iota$. For the contrary assume that there exists a $\agms(n)=\{a_{i,j}\}_{i,j=1}^{n}$ with the   magic constant $\mu=0$. Then on one hand $\sum_{i=1}^n\sum_{j=1}^na_{i,j}= 
\sum_{i=1}^n\mu=0$, and on the other $\sum_{i=1}^n\sum_{j=1}^na_{i,j}=\sum_{g\in \Gamma}g=\iota$, a contradiction.
\end{proof}

We summarize now our previous results in a single theorem.

\CoPthm*

\begin{proof}
    Recall that there is no $\Gamma$-magic square $\gms(2)$. Moreover by Observa\-tion~\ref{gr} we can assume that $\Gamma\in \gr$. If now $|\Gamma| = 2^{2\alpha}$, then we are done by Theorem~\ref{thm:order 2^s}. Therefore we can assume that  there exists a prime number $p>2$ that divides $n$. By Fundamental Theorem of Finite Abelian Groups the group  $\Gamma\cong H\oplus K$, where $|H|=p^{2\lambda}$ for some  integer $\lambda\ge 1$ and $K\in \gr$. By Theorem~\ref{thm:main-odd} there exists a $\ams{H}(p^{\lambda})$ and by Theorem~\ref{thm:Kotzig} there exists a {$K$-Kotzig array}  of size $p^{\lambda}\times |K|$. Thus  we apply Lemma~\ref{lem:SWL2}.\end{proof}

\section{Conclusion and Open Problems}\label{sec:conclusion}

In this paper we proved that a
non-trivial zero-sum $\Gamma$-magic square exists if and only if
$n>2$ and $\Gamma \in \gr$. Thus, the zero-sum case (that is, magic constant
$\mu = 0$) is now completely characterized.

Nevertheless, several natural questions remain open. One direction concerns the structure of possible magic constants.
Observe that in some cases a $\ms{\Gamma}(n)$ with a nonzero magic constant
$\mu$ can be transformed into a zero-sum $\ms{\Gamma}(n)$ by a uniform
translation. Indeed, if there exists an element $x \in \Gamma$ such that $-\mu = nx$, then adding $x$ to each entry of the square produces a~$\zms{\Gamma}(n)$.

\begin{exm}\label{exm: Z_27 x Z_3}
In Figure~\ref{additive MS}, for a $\ms{Z_9}(3)$ the magic constant is
$\mu=3$. Since $-\mu = 3x$ with $x = 2$, adding $2$ to all elements of
the array yields a zero-sum square $B=\ams{Z_9}(3)$.
\end{exm}

\begin{figure}[h]
\begin{subfigure}[t]{0.5\linewidth}
$$
A=\begin{array}{|c|c|c|}
\hline
	7& 0& 5\\\hline
	2& 4& 6\\\hline
	3& 8& 1\\\hline
\end{array}
$$
\caption{An $\ms{Z_9}(3)$ with $\mu=3$}
\end{subfigure}
\hfill
\begin{subfigure}[t]{0.5\linewidth}
$$
B=\begin{array}{|c|c|c|}
\hline
	0& 2& 7\\\hline
	4& 6& 8\\\hline
	5& 1& 3\\\hline
\end{array}
$$
\caption{A $\zms{Z_9}(3)$}
\end{subfigure}
\caption{An $\ms{Z_9}(3)$ and a $\ams{Z_9}(3)$}
\label{additive MS}
\end{figure}

However, such a translation is not always possible, as shown in
Figure~\ref{Z_2 x Z_8}.

\begin{figure}[H]
\begin{center}
\begin{tabular}{|c|c|c|c|}
\hline
    (0,0)  & (0,1) & (0,3) & (0,2)\\ \hline
    (0,7)  & (0,6) & (0,4) & (0,5)\\ \hline
    (1,0)  & (1,1) & (1,3) & (1,2)\\ \hline
    (1,7)  & (1,6) & (1,4) & (1,5)\\ \hline
\end{tabular}
\caption{$\ms{Z_2\oplus Z_8}(4)$ with the magic constant $\mu=(0,6)$}
\label{Z_2 x Z_8}
\end{center}
\end{figure}

This observation leads to the following natural problem.

\begin{oprb}
Let $\Gamma$ be an Abelian group of order $n^2>4$.
Determine all elements $\mu \in \Gamma$ for which there exists a
$\Gamma$-magic square of side $n$ with magic constant $\mu$.
\end{oprb}

For an Abelian group $\Gamma$ of order $n^2>4$, we can define the
\emph{spectrum of magic constants} by
\[
\Spec_{\Gamma}(n)
=
\left\{
\mu \in \Gamma \;:\;
\text{there exists a $\gms(n)$ with magic constant $\mu$}
\right\}.
\]

A natural problem is to characterize $\Spec_{\Gamma}(n)$ for arbitrary
Abelian groups $\Gamma$. In particular, it would be interesting to
determine whether $\Spec_{\Gamma}(n)$ has additional algebraic structure
(for example, whether it forms a subgroup or a coset of a subgroup of
$\Gamma$) and how it depends on the subgroup $n\Gamma = \{nx : x \in \Gamma\}$.

\section{Statements and Declarations}
The work of the first author was    supported by  AGH University of Krakow under grant no. 16.16.420.054,
funded by the Polish Ministry of Science and Higher Education.

\vskip1cm

\noindent







\begin{thebibliography}{00}

\bibitem{BCMMPW}
J. W. Brown, F. Cherry, L. Most, M. Most, E. T. Parker, W. D. Wallis, 
Completion of the spectrum of orthogonal diagonal Latin squares,
in \textit{Graphs, matrices, and designs,} 
43--49, 
Lecture Notes in Pure and Appl. Math., \textbf{139} (1993).

\bibitem{BurattiMerolaNakic2025} 
M.  Buratti, F. Merola, A. Nakic, Additive combinatorial designs. \textit{Des. Codes Cryptogr.} \textbf{93} (2025), 2717--2740.

\bibitem{Marco2}
M. Buratti, M. Galici, A. Montinaro, A. Nakic, A. Wassermann, $EA(q)$-additive Steiner 2-designs, Preprint arXiv:2511.01073 (2025).

\bibitem{Marco3} M. Buratti, A. Pasotti, Heffter spaces, \textit{Finite Fields Appl.} \textbf{98} (2024), Article number 102464.

\bibitem{Marco4} M. Buratti, A. Pasotti, Shiftable Heffter spaces, \textit{Des. Codes Cryptogr.} 93 (2025), 3863--3874.

   \bibitem{CaggegiFalconePavone2017} 
 A. Caggegi, G. Falcone, M. Pavone, On the additivity of block designs, \textit{J. Algebr. Comb.} \textbf{45} (2017), 271--294.



\bibitem{CicZ}  S. Cichacz,  {Zero sum partition into sets of the same order and its applications},   \textit{Electron. J.
Combin.} \textbf{25(1)} (2018), \#P1.20.

\bibitem{CicJaco}
S. Cichacz, Partition of Abelian groups into zero-sum sets by complete mappings and its application to the existence of a magic rectangle set, \textit{J. Algebr. Comb.} \textbf{61(24)} (2025) 112815

\bibitem{CicSur}
S. Cichacz, Disjoint zero-sum subsets in Abelian groups and theirs application -- survey, accepted for publication in \textit{Bolyai Society Mathematical Studies} (2025).

\bibitem{CicFro}
S. Cichacz, D. Froncek, {Magic squares on Abelian groups}, \textit{Discrete Math.} \textbf{349(7)} (2026), 115033.

\bibitem{Cichacz-Hincz-2}
S. Cichacz, T. Hinc,
A magic rectangle set on Abelian groups and its application,
\textit{Discrete Appl. Math.} \textbf{288} (2021), 201--210.


\bibitem{handbook}
C. J. Colbourn, J. H. Dinitz, eds.,
\textit{Handbook of combinatorial designs}, second edn.,
\textsc{Discrete Mathematics and its Applications (Boca Raton)}, 
Chapman \& Hill/CRC Press, Boca Raton, FL 2007.



\bibitem{ErdGinZiv}
P. Erdős, A. Ginzburg,  A. Ziv,  Theorem in the additive number theory, \textit{Bull. Res.
Council Israel} \textbf{10} (1961), 41–43.


\bibitem{Evans}
A. B. Evans, 
Magic rectangles and modular magic rectangles, 
\textit{J. Stat. Plann. Inference} \textbf{51} (1996), 171--180.














\bibitem{Skolem57}
T. Skolem, On certain distributions of integers in pairs with given differences, \textit{Mathematica
Scandinavica} \textbf{5} (1957), 57–68.

\bibitem{Sun-Yihui}
H. Sun, W. Yihui, 
Note on magic squares and magic cubes on Abelian groups, 
\textit{J. Math. Res. Exposition} \textbf{17(2)} (1997), 176--178.


\end{thebibliography}
\end{document}